\documentclass[12pt, twoside, leqno]{article}

\usepackage{amsmath,amsthm}
\usepackage{amssymb}

\usepackage{enumerate}
\usepackage{mathtools}

\usepackage{graphicx}

\theoremstyle{plain}
\newtheorem{theorem}{Theorem}
\numberwithin{theorem}{section}
\newtheorem{lemma}[theorem]{Lemma}

\newtheorem{proposition}[theorem]{Proposition}
\newtheorem{corollary}[theorem]{Corollary}

\theoremstyle{definition}

\newtheorem{remark}[theorem]{Remark}

\numberwithin{equation}{section}

\begin{document}


\baselineskip=12pt


\title{A characterisation of elementary abelian $3$-groups}

\author{C. S. Anabanti\footnote{The author is supported by a Birkbeck PhD Scholarship.}\\ \centering{\footnotesize The author dedicates this paper to Professor Sarah Hart with admiration and respect.}}

\date{}

\maketitle

\renewcommand{\thefootnote}{}



\renewcommand{\thefootnote}{\arabic{footnote}}
\setcounter{footnote}{0}

\begin{abstract}
\noindent T{\u a}rn{\u a}uceanu [Archiv der Mathematik, \textbf{102 (1)},  (2014), 11--14] gave a characterisation of elementary abelian $2$-groups in terms of their maximal sum-free sets. His theorem states that a finite group $G$ is an elementary abelian $2$-group if and only if the set of maximal sum-free sets coincides with the set of complements of the maximal subgroups. 
A corollary is that the number of maximal sum-free sets in an elementary abelian $2$-group of finite rank $n$ is $2^n-1$.
Regretfully, we show here that the theorem is wrong. 
We then prove a correct version of the theorem from which the desired corollary can be deduced.  
Moreover, we give a characterisation of elementary abelian $3$-groups in terms of their maximal sum-free sets. 
A corollary to our result is that the number of maximal sum-free sets in an elementary abelian $3$-group of finite rank $n$ is $3^n-1$.
Finally, for prime $p>3$ and $n\in \mathbb{N}$, we show that there is no direct analogue of this result for elementary abelian $p$-groups of finite rank $n$.
\end{abstract} 

\section{Preliminaries} 
The well-known result of Schur which says that whenever we partition the set of positive integers into a finite number of parts, at least one of the parts contains three integers $x,y$ and $z$ such that $x+y=z$ introduced the study of sum-free sets. Schur \cite{S1917} gave the result while showing that the Fermat's last theorem does not hold in $\mathbb{F}_p$ for sufficiently large $p$. The result was later extended to groups as follows: A non-empty subset $S$ of a group $G$ is sum-free if for all $s_1,s_2 \in S$, $s_1s_2\notin S$. (Note that the case $s_1=s_2$ is included in this restriction.)
An example of a sum-free set in a finite group $G$ is any non-trivial coset of a subgroup of $G$. Sum-free sets have applications in Ramsey theory and are also closely related to the widely studied concept of caps in finite geometry. Some questions that appear interesting in the study of sum-free sets are:
(i) How large can a sum-free set in a finite group be?
(ii) Which finite groups contain maximal by inclusion sum-free sets of small sizes?
(iii) How many maximal by cardinality sum-free sets are there in a given finite group?
Each of these questions has been attempted by several researchers; though none is fully answered. For question (i), Diananda and Yap \cite{DY1969}, in 1969, following an earlier work of Yap \cite{Y1969}, determined the sizes of maximal by cardinality sum-free sets in finite abelian groups $G$, where $|G|$ is divisible by a prime $p \equiv 2$ mod $3$, and where $|G|$ has no prime factor $p\equiv 2$ mod $3$ but $3$ is a factor of $|G|$. They gave a good bound in the case where every prime factor of $|G|$ is congruent to $1~mod~3$. Green and Rusza \cite{GR2005} in 2005 completely answered question (i) in the finite abelian case. The question is still open for the non-abelian case, even though there has been some progress by Kedlaya \cite{K1997,K1998}, Gowers \cite{G2008}, among others. 
For question (ii), Street and Whitehead \cite{SW1974} began research in that area in 1974. They called a maximal by inclusion sum-free set, a locally maximal sum-free set (LMSFS for short), and calculated all LMSFS in groups of small orders, up to $16$ in \cite{SW1974,SW1974A} as well as a few higher sizes.  
In 2009, Giudici and Hart \cite{GH2009} started the classification of finite groups containing LMSFS of small sizes. Among other results, they classified all finite groups containing LMSFS of sizes $1$ and $2$, as well as some of size $3$. The size $3$ problem was resolved in \cite{AH2016}. Question (ii) is still open for sizes $k\geq 4$.\\
\\
To be consistent with our notations, we will use the term `maximal' to mean `maximal by cardinality' and `locally maximal' to mean `maximal by inclusion'. 
T{\u a}rn{\u a}uceanu \cite{T2014} in 2014 gave a characterisation of elementary abelian $2$-groups in terms of their maximal sum-free sets. His theorem (see Theorem 1.1 of \cite{T2014}) states that ``a finite group $G$ is an elementary abelian $2$-group if and only if the set of maximal sum-free sets coincides with the set of complements of the maximal subgroups". The author of \cite{T2014} didn't define the term maximal sum-free sets.
Unfortunately, the theorem is false whichever definition is used.
If we take `maximal' in the theorem to mean `maximal by cardinality', then a counter example is the cyclic group $C_4$ of order $4$, given by $C_4=\langle x \mid x^4=1\rangle$. Here, there is a unique maximal (by cardinality) sum-free set namely $\{x,x^3\}$, and it is the complement of the unique maximal subgroup. But $C_4$ is not elementary abelian. On the other hand, if we take `maximal' to mean `maximal by inclusion', then the theorem will still be wrong since $S=\{x_1,x_2,x_3,x_4,x_1x_2x_3x_4\}$ is a maximal by inclusion sum-free set in $C_2^4=\langle x_1,x_2,x_3,x_4 \mid x_i^2=1,x_ix_j=x_jx_i \text{ for } 1 \leq i,j \leq 4 \rangle$, but does not coincide with any complement of a maximal subgroup of $C_2^4$. 
\\\\
For a prime $p$ and $n\in \mathbb{N}$, we write $\mathbb{Z}_p^n$ for the elementary abelian $p$-group of finite rank $n$.
We recall here that the number of maximal subgroups of $\mathbb{Z}_p^n$ is $\sum\limits_{k=0}^{n-1}p^k$. In this paper, we give a correction to Theorem 1.1 of \cite{T2014} which will then make its desired corollary hold. For the rest of this section, we state the main result of this paper and its immediate corollary. Recall that $\Phi(G)$ is the Frattini subgroup of $G$.
\begin{theorem}\label{thm1} 
A finite group $G$ is an elementary abelian $3$-group if and only if the set of non-trivial cosets of each maximal subgroup of $G$ coincides with two maximal sum-free sets in $G$, every maximal sum-free set is a non-trivial coset of a maximal subgroup, and $\Phi(G)=1$.
\end{theorem}

\begin{corollary}
The number of maximal sum-free sets in $\mathbb{Z}_3^n$ is $3^n-1$.
\end{corollary}
\begin{proof}
As the number of maximal subgroups of $\mathbb{Z}_3^n$ is $\frac{3^n-1}{2}$, it follows immediately from Theorem \ref{thm1} that the number of maximal sum-free sets in $\mathbb{Z}_3^n$ is $2(\frac{3^n-1}{2})=3^n-1$.
\end{proof}

\section{Main results}
\noindent Let $S$ be a sum-free set in a finite group $G$. We define $SS=\{xy \mid x,y \in S\}$,  $S^{-1}=\{x^{-1} \mid x\in S\}$ and $SS^{-1}=\{xy^{-1} \mid x,y \in S\}$. Clearly, $S\cap SS=\varnothing$. Moreover, $S\cap SS^{-1}=\varnothing$ as well; for if $x,y,z\in S$ with $x=yz^{-1}$, then $xz=y$, contradicting the fact that $S$ is sum-free.

\subsection{Correction to Theorem 1.1 of \cite{T2014}}
We begin with a remark that what is missing in the statement of Theorem 1.1 of \cite{T2014} is the assumption that $\Phi(G)=1$, where $\Phi(G)$ denotes the Frattini subgroup of $G$. A correction to Theorem 1.1 of \cite{T2014} is the following (from where the suggested corollary holds):

\begin{theorem}[The Correction]\label{T1}
A finite group $G$ is an elementary abelian $2$-group if and only if the set of maximal sum-free sets coincides with the set of complements of the maximal subgroups, and $\Phi(G)=1$. 
\end{theorem} 

\begin{remark}\label{R1}
(a) Let $G$ be a finite group and $S$ a sum-free set in $G$. For $x_1\in S$, define $x_1S:=\{x_1x_2|x_2 \in S\}$. As $|x_1S|=|S|$ and $S \cup x_1S \subseteq G$, with $S \cap x_1S=\varnothing$, we have that $2|S|\leq |G|$; so $|S|\leq \frac{|G|}{2}$. This shows that the size of a sum-free set in $G$ is at most $\frac{|G|}{2}$.\\
(b) We recall Lemma 3.1 of \cite{GH2009} which says that a sum-free set $T$ in a finite group $G$ is locally maximal if and only if $G=T \cup TT \cup TT^{-1} \cup T^{-1}T \cup \sqrt{T}$, where $\sqrt{T}=\{x\in G \mid x^2 \in T\}$. 
Now, let $S$ be a maximal sum-free set in $G=\mathbb{Z}_2^n$. As every maximal sum-free set is locally maximal and $SS=SS^{-1}=S^{-1}S$, with $\sqrt{S}=\varnothing$, Lemma 3.1 of \cite{GH2009} yields that $G=S\dot\cup SS$.
\end{remark}

We now give a proof of Theorem \ref{T1}

\begin{proof}
Let $G=\mathbb{Z}_2^n$, and $N$ be a maximal subgroup of $G$. Clearly, $|N|=\frac{|G|}{2}$. Let $M$ be the non-trivial coset of $N$ in $G$. Then $M$ is sum-free of size $\frac{|G|}{2}$ in $G$. By Remark \ref{R1}(a) therefore, $M$ is a maximal sum-free set in $G$. So each maximal subgroup of $G$ has its complement as a maximal sum-free set in $G$. 
Next, we show that every maximal sum-free set in $G$ is the complement of a maximal subgroup of $G$. Let $S$ be a maximal sum-free set in $G$, and let $x\in S$ be arbitrary. From $xS \subseteq SS$, we obtain that $|xS|\leq |SS|$, and from Remark \ref{R1}(b) that $G=S\dot\cup SS$ and the fact that $|S|=\frac{|G|}{2}$, we obtain that $|SS|\leq |G|-|S|=|S|=|xS|$.
Therefore $xS=SS$, and $G=S\dot\cup xS$. Define $H:=xS$. 
To show that $H$ is a subgroup of $G$, we simply show that $H$ is closed. 
Let $a$ and $b$ be elements of $H$. Then $a=xy$ and $b=xz$ for some $y,z\in S$. So $ab=yz\not\in S$. Hence $ab\in H$, and $H$ is closed. Thus $H$ is a subgroup of $G$. The fact that $H$ is a maximal subgroup of $G$ follows from the definition of $H$. Clearly, $S$ is the complement of $H$ in $G$ as desired.
The last part of the result that $\Phi(G)=1$ follows from the fact that the intersection of maximal subgroups of $G$ is trivial.
For the converse, suppose $G$ is a finite group such that the set of maximal sum-free sets in $G$ are precisely the complements of the maximal subgroups of $G$, and $\Phi(G)=1$. Remark \ref{R1}(a) tells us that any maximal sum-free set in $G$ has size at most $\frac{|G|}{2}$. Therefore the complement of the maximal subgroups must have size at most $\frac{|G|}{2}$, and hence every maximal subgroup is of index $2$ in $G$. Now, let $R$ be a Sylow $2$-subgroup of $G$. If $G$ is not a $2$-group, then $R$ is contained in a maximal subgroup of $G$ whose index must be odd; a contradiction. Therefore $G$ is a $2$-group. It is a basic result in group theory that for a $p$-group $P$, the quotient $P/\Phi(P)$ is always elementary abelian. As $\Phi(G)=1$, we conclude that $G$ is an elementary abelian $2$-group. 
\end{proof}

\subsection{Proof of Theorem \ref{thm1}}

\begin{lemma}\label{N1}
Let $S$ be sum-free in $G=\mathbb{Z}_3^n$ ($n\in \mathbb{N}$), and let $x\in S$. Then the following hold:\\
(i) any two sets in $\{S,x^{-1}S,xS\}$ are disjoint;
(ii) any two sets in $\{S,SS^{-1},S^{-1}\}$ are disjoint.\\
Moreover, if $S$ is maximal, then the following also hold:\\
(iii) $S \cup x^{-1}S \cup xS=G$ and $|S|=\frac{|G|}{3}$;
(iv) $S\cup SS^{-1} \cup S^{-1}=G$.
\end{lemma}

\begin{proof}
(i)  As $S$ is sum-free, $S\cap xS=\varnothing=S\cap x^{-1}S$. So we only need to show that $xS \cap x^{-1}S = \varnothing$. Suppose for contradiction that $xS \cap x^{-1}S \neq \varnothing$. Then there exist $y,z \in S$ such that $xy=x^{-1}z$. This means that $y=xz$; a contradiction. Therefore $xS \cap x^{-1}S= \varnothing$. The proof of (ii) is similar to (i). For (iii), as $S \cup x^{-1}S \cup xS \subseteq G$, we have that $3|S|\leq |G|$; whence $|S|\leq \frac{|G|}{3}$. 
Each maximal subgroup of $G$ has size $\frac{|G|}{3}$. As any non-trivial coset of such a subgroup is sum-free and has size $\frac{|G|}{3}$; such a coset of the maximal subgroup must be maximal sum-free. Thus, $|S|=\frac{|G|}{3}$, and $S \cup x^{-1}S \cup xS=G$. The proof of (iv) is similar.
\end{proof}


\begin{proposition}\label{N2}
Suppose $S$ is a maximal sum-free set in an elementary abelian $3$-group $G$, and let $x\in S$. Then the following hold: (i) $x^{-1}S=S^{-1}S$; (ii) $xS=S^{-1}=SS$.
\end{proposition}

\begin{proof}
Let $S$ be a maximal sum-free set in an elementary abelian $3$-group $G$, and $x\in S$.\\
(i) Clearly, $x^{-1}S \subseteq S^{-1}S$; therefore $|x^{-1}S| \leq |S^{-1}S|$. By Lemma \ref{N1}(iv), $|S^{-1}S|\leq |G|-(|S|+|S^{-1}|)=3|S|-2|S|=|S|=|x^{-1}S|$. Therefore, $|x^{-1}S|=|S^{-1}S|$; whence $x^{-1}S=S^{-1}S$.\\
(ii) Let $y\in xS$. By Lemma \ref{N1}(i) and Proposition \ref{N2}(i), we have that $y\not\in (S\dot\cup SS^{-1})$. So Lemma \ref{N1}(iv) tells us that $y\in S^{-1}$, and we conclude that $xS\subseteq S^{-1}$. On the other hand, if $y\in S^{-1}$, then Lemma \ref{N1}(ii), Proposition \ref{N2}(i) and Lemma \ref{N1}(iii) yield $y\in xS$; so $S^{-1}\subseteq xS$. Therefore $xS=S^{-1}$. Now, 
\begin{equation}\label{E1}
SS=\bigcup_{x\in S}xS=\bigcup_{x\in S}S^{-1}=S^{-1}.
\end{equation}
Thus, $xS=S^{-1}=SS$ as required.
\end{proof}
 
\noindent Suppose $p$ is the smallest prime divisor of the order of a finite group $G$, and $H$ is a subgroup of index $p$ in $G$. Then $H$ is normal in $G$. This fact is well-known but we include a short proof for the reader's convenience.
Suppose for a contradiction that $H$ is not normal. 
Then for some $g\in G$, we have $H^g\neq H$. But $|H^gH|=\frac{|H^g||H|}{|H^g \cap H|}$ $=\frac{|H|^2}{|H^g \cap H|}$ $=|H|\frac{|H|}{|H ^g\cap H|}$ $\geq$ $|H|p=|G|$; thus $H^gH=G$. Therefore, $g=(gh_1g^{-1})h_2$ for some $h_1,h_2\in H$. So $g=h_2h_1\in H$, and we conclude that $H^g=H$; a contradiction. Therefore $H$ is normal in $G$.
\\

We now give a proof of Theorem \ref{thm1}
\begin{proof}
Let $G$ be an elementary abelian $3$-group of finite rank $n$. 
Clearly, every maximal subgroup of $G$ has size $3^{n-1}$; so is associated with two non-trivial cosets, which are maximal sum-free sets. Next, we show that every maximal sum-free set in $G$ is a non-trivial coset of a maximal subgroup of $G$. Suppose $S$ is a maximal sum-free set in $G$. Let $x\in S$ be arbitrary, and define $H:=x^{-1}S$. We show that $H$ is a subgroup of $G$. To do this, we show that $H$ is closed. 
Let $a$ and $b$ be elements of $H$. Then $a=x^{-1}y$ and $b=x^{-1}z$ for some $y,z\in S$. 
Since $ab=x^{-1}(x^{-1}yz)$, it is sufficient to show that $x^{-1}yz\in S$. Recall from Lemma \ref{N1}(iii) that $G=S\cup x^{-1}S\cup xS$. From Proposition \ref{N2}(ii) therefore, $G=S\cup x^{-1}S\cup S^{-1}$. 
Now, suppose $x^{-1}yz\in x^{-1}S$. Then there exists $q\in S$ such that $x^{-1}yz= x^{-1}q$. This implies that $yz=q$; a contradiction.
Next suppose $x^{-1}yz\in S^{-1}$. Then there exists $q\in S$ such that $x^{-1}yz= q^{-1}$. So $yz=xq^{-1}$, and we obtain that $x^{-1}q=y^{-1}z^{-1}=(yz)^{-1}$; a contradiction as $x^{-1}q\in x^{-1}S$, $(yz)^{-1}\in (SS)^{-1}=S$ by Equation \ref{E1}, and Lemma \ref{N1}(i) tells us that $x^{-1}S\cap S=\varnothing$. 
We have shown that $x^{-1}yz\not\in x^{-1}S\cup S^{-1}$. In the light of $G=S\cup x^{-1}S\cup S^{-1}$ therefore, $x^{-1}yz\in S$; whence, $H$ is closed. So $H$ is a subgroup of $G$.
As $|H|=|x^{-1}S|=|S|=\frac{|G|}{3}$, we conclude that $H$ is a maximal subgroup of $G$, and $S=xH$ is a non-trivial coset of $H$ in $G$. So we have shown now that every maximal sum-free set in $G$ is a non-trivial coset of a maximal subgroup of $G$. The third part that $\Phi(G)=1$ follows from the fact that the intersection of maximal subgroups of $G$ is trivial.
Conversely, suppose $G$ is a finite group such that the set of non-trivial cosets of each maximal subgroup of $G$ coincides with two maximal sum-free sets in $G$, every maximal sum-free set of $G$ is a coset of a maximal subgroup of $G$, and $\Phi(G)=1$.
First and foremost, $G$ has no subgroup of index $2$; otherwise it will have a maximal sum-free set which is not a coset of a subgroup of index $3$.
As the smallest index of a maximal subgroup of $G$ is $3$, any such subgroup must be normal in $G$.
Let $H$ be a Sylow $3$-subgroup of $G$. Then either $H=G$ or $H$ is contained in a maximal subgroup (say $M$) of $G$. Suppose $H$ is contained in such maximal subgroup $M$. As $|G/M|=3$, we deduce immediately that $|G:H|$ is divisible by $3$; a contradiction! Therefore, $H=G$, and we conclude that $G$ is a $3$-group. Now, $G$ is an elementary abelian $3$-group follows from the fact that $\Phi(G)=1$ and $P/\Phi(P)$ is elementary abelian for every $p$-group $P$.
\end{proof}

\noindent In conclusion, if $G=\mathbb{Z}_p^n$ for prime $p>3$ and $n\in \mathbb{N}$, then there exists a normal subgroup $N$ of $G$ such that $G/N\cong C_p$, and $C_p$ has a maximal sum-free set of size at least $2$ (the latter fact follows from the classification of groups containing maximal by inclusion sum-free sets of size $1$ in \cite[Theorem 4.1]{GH2009}). 
The union of non-trivial cosets of $N$ corresponding to this maximal sum-free set of $C_p$ is itself sum-free in $G$.
So $G$ has a maximal sum-free set of size at least $2|N|$. This argument shows that a direct analogue of Theorem \ref{thm1} is not possible for elementary abelian $p$-groups, where $p>3$ and prime.


\bigskip
\noindent \textbf{Chimere Stanley Anabanti}\\
Birkbeck, University of London\\
c.anabanti@mail.bbk.ac.uk
\end{document}